\documentclass[9pt]{extarticle}

\usepackage{amsmath, graphicx}
\usepackage{amssymb}
\usepackage{amsfonts}
\usepackage{amscd}
\usepackage{xcolor}  
\usepackage{amsthm}
\usepackage{tikz}
\usepackage[top=1.12in, bottom=1.27in, left=1.38in, right=1.38in]{geometry}
\usepackage{layout}
\usepackage{makecell}
\usepackage{multirow}

\usepackage[hidelinks,draft=false]{hyperref}
\usepackage{bookmark}

\bookmarksetup{
  open,
  numbered,
  addtohook={%
    \ifnum\bookmarkget{level}>2 
      \renewcommand*{\numberline}[1]{}%
    \fi
  },
}

\usepackage{fancybox}
\usepackage{dsfont}
\usepackage{chngcntr}
\counterwithin*{equation}{section}
\usepackage{bm}
\usepackage{mathtools} 
\usepackage{stmaryrd} 
\usepackage{extarrows} 
\usepackage{cancel}

\input{xypic}
\xyoption{all}

\usepackage{threeparttable}
\usepackage{array}
\usepackage{booktabs}
\usepackage{blindtext}
\usepackage{mathrsfs}  

\usepackage{enumitem}
\usepackage{bbm} 
\usepackage{setspace} 
\usepackage{scalerel} 

\usepackage[normalem]{ulem} 

\usepackage{adjustbox}

\renewcommand{\theequation}{\arabic{section}.\arabic{equation}}
\newtheorem{theorem}{Theorem}[section]
\newtheorem{prop}[theorem]{Proposition}

\newtheorem{lemma}[theorem]{Lemma}
\newtheorem{cor}[theorem]{Corollary}

\newtheorem{defn}[theorem]{Definition}

\newtheorem{remark}[theorem]{Remark}
\newenvironment{rem}{\begin{remark}\rm}{\end{remark}}

\newtheorem{facts}[theorem]{Fact}

\newtheorem{example}[theorem]{Example}

\newtheorem{exercise}[theorem]{Exercise}

\newtheorem{terminology}[theorem]{Terminology}

\newtheorem{notation}[theorem]{Notation}

\newtheorem{observation}[theorem]{Observation}

\newtheorem{question}[theorem]{Question}
\newenvironment{que}{\begin{question}\rm}{\end{question}}

\makeatletter
\DeclareRobustCommand*\uell{\mathpalette\@uell\relax}
\newcommand*\@uell[2]{
  \setbox0=\hbox{$#1\ell$}
  \setbox1=\hbox{\rotatebox{15}{$#1\ell$}}
  \dimen0=\wd0 \advance\dimen0 by -\wd1 \divide\dimen0 by 2
  \mathord{\lower 0.1ex \hbox{\kern\dimen0\unhbox1\kern\dimen0}}
}

\def\a{\alpha}
\def\b{\beta}
\def\ga{\gamma}
\def\d{\delta}
\def\la{\lambda}
\def\s{\sigma}
\def\t{\theta}

\def\ep{\epsilon}

\def\L{\Lambda}
\def\Ga{\Gamma}
\def\Om{\Omega}
\def\T{\Theta}

\def\CC{\mathcal C}

\def\D{\mathcal D}
\def\A{\mathcal A}

\def\C{\mathbb C}

\def\Z{\mathbb Z}
\def\R{\mathbb R}

\def\S{\mathbb S}

\def\on{\operatorname}

\def\dim{\on{dim}}
\def\rank{\on{rank}}
\def\lb{\langle}
\def\rb{\rangle}

\def\xr{\xrightarrow}

\def\ov{\overline}

\def\non{\noindent}
\def\ds{\displaystyle}

\def\wt{\widetilde}
\def\wh{\widehat}

\def\p{\prime}

\def\spinc{ {\on{spin}^{{\rm c}} }}

\def\sl{\mathfrak{sl}}

\def\su{{ \frak{su} }}

\def\Cl{\on{Cl}}

\def\Hom{\on{Hom}}

\def\End{\on{End}}
\def\Sym{\on{Sym}}

\def\c{{ \rm c}}

\def\ind{\on{ind}}
\def\tr{\on{tr}}

\def\pa{\partial}

\def\ul{\underline}
\def\ov{\overline}

\newcommand{\sdfrac}[2]{\mbox{\small$\displaystyle\frac{#1}{#2}$}}
\newcommand{\fdfrac}[2]{\mbox{\footnotesize$\displaystyle\frac{#1}{#2}$}}

\def\w{\wedge}

\def\td{\on{td}}
\def\ch{\on{ch}}

\def\CCl{\C\!\on{l}}

\usepackage{scalerel,stackengine}
\newcommand\pig[1]{\scalerel*[5.5pt]{\Big#1}{%
  \ensurestackMath{\addstackgap[1.5pt]{\big#1}}}}

\usepackage{graphicx}

\def\vt{\ov{\t}}
\def\vm{\ov{\mu}}

\def\bt{\vartheta}
\def\vbt{\ov{\bt}}
\def\codim{\on{codim}}

\title{{\bf Localizing a Dirac operator via $J$-holomorphic curves}}

\author{Junho Lee}

\date{\empty}
\addtocounter{section}{0}
\begin{document}

\maketitle

\begin{abstract}

Let $(M,J)$ be a compact almost hermitian $4$-manifold with a smooth embedded $J$-holomorphic curve $C$ representing the canonical class.
Motivated by the symplectic Bogomolov--Miyaoka--Yau conjecture, we choose a twisted $\spinc$ Dirac operator
$\D$ on $M$ satisfying
$$
\ind\D=3c_2(M)-c_1^2(M).
$$
We then use the curve $C$ to construct a complex-linear perturbation $\A$ of $\D$ whose singular set is $Z_\a\sqcup C$,
where  $Z_\a=\a^{-1}(0)$ for a transverse section $\a$ of $\L^{1,0}M$ nonvanishing on $C$.
Applying Maridakis' index localization theorem, we express $\ind\D$ as the sum of localized contributions:
$3c_2(M)$ from $Z_\a$ and $-c_1^2(M)$ from $C$.
\end{abstract}

\section{Introduction}

Localization is a fruitful idea in geometry and topology.
One analytic example is a concentration principle for elliptic operators. Let
$$
\D:\Ga(E^+)\to \Ga(E^-)
$$
be a first-order elliptic operator on a compact Riemannian manifold $M$, and consider a family
$$
\D_s=\D+s\A,\qquad s\in\R,
$$
where $\A:E^+\to E^-$ is a bundle map. Maridakis~\cite{M1} showed that if $\A$ satisfies the concentration condition
\begin{equation}\label{CC}
\s_{\D^*}(\xi)\circ \A+\A^*\circ \s_\D(\xi)=0
\quad \forall \xi\in T^*M,
\end{equation}
where $\s_\D$ denotes the principal symbol of $\D$, then as $s\to \infty$ the kernels of $\D_s$ and $\D_s^*$
(as well as the elements of the low eigenspaces of $\D_s^*\D_s$ and $\D_s\D_s^*$) concentrate near the singular set
$$
Z_\A:=\{x\in M:\ker \A_x\ne 0\}
$$
and decay exponentially in the distance to $Z_\A$.
Under additional hypotheses, this concentration leads to a localization theorem for the index of $\D$.

The concentration principle has appeared in various contexts; see, for example, \cite{W, T1, R, PR, LP, M1, GW, M2, P, Nag, L}.
The prototype is Witten's deformation of the de Rham complex~\cite{W}. Another example is Taubes' ingenious localization proof of
the Riemann--Roch theorem~\cite{T1}. For Dirac operators, Prokhorenkov and Richardson~\cite{PR} studied complex-linear perturbations and proved an index localization theorem when the singular set is finite. Maridakis later formulated the general concentration principle described above for real elliptic operators~\cite{M1} and
developed a localization theorem for singular sets whose components are compact submanifolds~\cite{M2}.

This paper provides an explicit example of Maridakis' index localization theorem.
The original motivation was to investigate whether the concentration principle could be used to study the symplectic
Bogomolov--Miyaoka--Yau (BMY) conjecture. This well-known conjecture asserts that
a compact minimal symplectic $4$-manifold $M$ with $c_1^2(M)\geq 0$ satisfies
$$
3c_2(M)-c_1^2(M)\geq 0.
$$
In this context, a natural question is:

\begin{que}
Can we find an elliptic operator $\D$ and a bundle map $\A$ satisfying
\eqref{CC} and
\begin{equation}\label{indBMY}
\ind\D=3c_2(M)-c_1^2(M),
\end{equation}
such that as $s\to\infty$, the concentration of $\ker\D_s$ and $\ker\D_s^*$ along $Z_\A$
can be used to prove $\ind\D\geq 0$, and hence the BMY inequality?
\end{que}

Maridakis' index localization theorem~\cite{M2} provides a possible framework for this approach.
Under suitable hypotheses, the theorem shows that the normal derivatives of $\A$ along $Z_\A$ determine eigenbundles over $Z_\A$.
The present paper grew out of attempts to understand the role of these eigenbundles
in the setting of the BMY inequality, but it does not address the inequality.
Since our construction does not require a symplectic structure, we work in the almost hermitian setting.

Let $(M,J)$ be a compact almost hermitian $4$-manifold, and let $\S=\S^+\oplus\S^-$ be the canonical $\spinc$ spinor bundle.
Given a $\Z_2$-graded hermitian vector bundle $W=W^+\oplus W^-$, one obtains an associated Dirac operator $\D:\Ga(E^+)\to \Ga(E^-)$,
where $E=\S\otimes W$.
The Atiyah--Singer index theorem provides a way to choose bundles $W$ so that $\D$ satisfies \eqref{indBMY}.

Once $W$ is fixed, a construction of Prokhorenkov and Richardson~\cite{PR} (see Section~\ref{CPair})
associates to each complex-linear bundle map $\psi:W^+\to W^-$ a perturbation $\A$ of $\D$ satisfying the concentration condition~\eqref{CC}.
There is no obvious method for choosing $\psi$ so that the associated perturbation $\A$
has the desired singular set  and the needed geometric properties.
The main challenge is therefore to construct such a bundle map $\psi$.

Now assume that the canonical class of $M$ is represented by a smooth embedded $J$-holomorphic curve $C$.
Such curves exist for a broad class of almost hermitian $4$-manifolds,
including (i) compact complex surfaces with a smooth canonical divisor
and, by a famous theorem of Taubes, (ii) compact symplectic $4$-manifolds with $b^+>1$ (cf. \cite{T2}).
Given such a $C$, we use it to construct a complex-linear bundle map $\psi$ and analyze the associated perturbation $\A$.
The main features of the construction are as follows.
\begin{itemize}
\item[(F1)]
The singular set of the perturbation $\A$ is
$$
Z_\A=Z_\a\cup C,
$$
where $Z_\a=\a^{-1}(0)$ for a transverse section $\a$ of $\L^{1,0}M$.
Moreover, $\A$ vanishes on $Z_\A$, and hence
$$
\ker\big(\A|_{Z_{\A}}\big)=E^+|_{Z_{\A}}\qquad\text{and}\qquad\ker\big(\A^*|_{Z_{\A}}\big)=E^-|_{Z_{\A}}.
$$
In particular, both kernels form vector bundles over each component of $Z_\A$.
\end{itemize}

Note that the kernels of $\D_s$ and $\D_s^*$ concentrate near $Z_\A$ as $s\to\infty$,
and hence $\ind\D=3c_2(M)-c_1^2(M)$ is localized on $Z_\A$.
By construction, the location of $Z_\a$ can be chosen freely.
The two cases of primary interest are $Z_\a\cap C=\emptyset$ and $Z_\a\subset C$.
In this paper, we consider only the former case by choosing $\a$ that does not vanish on $C$.
The remaining features concern the case $Z_\a\cap C=\emptyset$.

Applying Maridakis' Theorem to the pair $(\D,\A)$ along $C$, we obtain:

\begin{itemize}
\item[(F2)]
The normal derivatives of $\A$ along each component $Z$ of $C$, together with Clifford multiplication on $E=\S\otimes W$, define endomorphisms
$\CC^\pm_Z$ of $E^\pm|_Z$ (see \eqref{CCoperator}). The eigenspaces of $\CC^\pm_Z$ form eigenbundles over $Z$.
We identify the eigenbundles entering the localization formula explicitly as complex line bundles. They carry associated Dirac operators
on $Z$, and the sum of the indices of these Dirac operators over all components $Z$ of $C$ is $-c_1^2(M)$.
This is the localized contribution from $C$ to $\ind\D$.
\end{itemize}

Unfortunately, we cannot directly apply Maridakis' theorem to $(\D,\A)$ at each point of $Z_\a$.
The non-degeneracy hypothesis (Condition~(C2) in Section~\ref{ILT}) at each point of $Z_\a$ requires that
the common complex rank of $W^+$ and $W^-$ be even (see Lemma~\ref{Obstruction}), whereas $W^\pm$ have complex rank $3$.
We get around this by passing to a stabilization $(\wh\D, \wh\A)$.

Applying Maridakis' theorem to
the pair $(\wh\D,\wh\A)$, we obtain:

\begin{itemize}
\item[(F3)]
After a suitable local deformation of $\wh\A$ near each point of $Z_\a$,
Maridakis' index localization theorem expresses the index $\ind\wh \D=\ind\D$ as the sum of the localized contributions from $Z_\a$ and $C$.
The contribution to $\ind\wh\D$ from $Z_\a$ is $3c_2(M)$, while the contribution from $C$ is still the sum of the indices in (F2), namely $-c_1^2(M)$.
\end{itemize}

The factor 3 in $3c_2(M)$ arises from the homotopy classes of the maps
$$
\psi|_{\pa D^4}: \pa D^4\cong S^3\cong SU(2)\to GL(3,\C),
$$
for small disks $D^4$ around each point of $Z_\a$ (see Section~\ref{RHC}).

The paper is organized as follows.
Section~2 constructs the pair $(W,\psi)$ and its associated
concentrating pair $(\D,\A)$ and establishes (F1).
Section~3 recalls Maridakis' index localization theorem.
Sections~4 and~5 establish (F2) and (F3), respectively.
The appendix discusses the relative homotopy groups used in Section~\ref{ContributionC0}.

\medskip
\non
{\bf Acknowledgements:} The author is grateful to Thomas H. Parker for many insightful discussions and helpful comments on the paper,
and to Manousos Maridakis for clarifying questions regarding his index localization theorem and useful comments.
I also thank Paul Feehan for his helpful comments.

\medskip
\non
{\bf AI disclosure:}
The author used ChatGPT 5.6 Sol to assist in constructing the local deformation needed to compute the contribution from $Z_\a$.
It suggested the explicit isomorphism $\pi_3(GL(n,\C))\cong \Z$ presented in the Appendix and
the map $R$ in \eqref{DefmapR}, both of which are built from the group $SU(2)$.
These resolved the local deformation problem near $Z_\a$.
The author verified and developed the resulting argument and takes full responsibility for the contents of the paper.

\section{Main construction}

Throughout this paper, $(M,J)$ denotes a compact almost hermitian $4$-manifold with
a smooth embedded, possibly disconnected, $J$-holomorphic curve $C$ representing the canonical class.
In this section, we recall from \cite{PR} the general form of complex-linear perturbations satisfying the concentration condition,
construct a pair $(\D,\A)$ adapted to $C$, and prove (F1) stated in the Introduction.

The almost complex structure $J$ determines a canonical $\spinc$ structure on $M$, whose spinor bundle is
$\S=\S^+\oplus\S^-$, where
\begin{equation}\label{spinorbdle}
\S^+=\L^{0,0}M\oplus\L^{0,2}M
\qquad\text{and}\qquad
\S^-=\L^{0,1}M.
\end{equation}
This $\spinc$ structure comes with Clifford multiplication
$\c(\xi):\S^\pm\to\S^\mp$ for $\xi\in T^*M$.
Let $\ga$ be the chirality operator on $\S$, i.e., $\ga|_{\S^\pm}=\pm Id$.

\subsection{Concentration pairs}
\label{CPair}

We begin by recalling how a bundle map gives rise to a concentration pair \cite{PR}.
Let $W=W^+\oplus W^-$ be a $\Z_2$-graded hermitian vector bundle over $M$. The tensor product $E_W:=\S\otimes W$ carries
the induced $\Z_2$-grading
\begin{equation}\label{Tbdle}
E_W^+=(\S^+\otimes W^+)\oplus(\S^-\otimes W^-)
\qquad\text{and}\qquad
E_W^-=(\S^+\otimes W^-)\oplus(\S^-\otimes W^+).
\end{equation}
The Clifford action on $E_W$ is $\c_{E_W}(\xi):=\c(\xi)\otimes Id_W$. Its restrictions  $E_W^\pm\to E_W^\mp$ are given by
$$
\c_{E_W}(\xi)=
\begin{bmatrix}
0 & \c(\xi)\otimes Id_{W^\mp}
\\[3pt]
\c(\xi)\otimes Id_{W^\pm} & 0
\end{bmatrix}: E_W^\pm\to E_W^\mp.
$$
A $\spinc$ connection on $\S$, induced from the Levi-Civita connection on $M$ and a hermitian connection
on $K_M^{-1}$, preserves the chirality grading in \eqref{spinorbdle}. Equip $W=W^+\oplus W^-$ with the direct sum of hermitian connections on $W^\pm$.
The resulting tensor product connection on $E_W=\S\otimes W$ preserves the four summands in \eqref{Tbdle} and is compatible with the Clifford multiplication $\c_{E_W}$.

Let $\D_{W^\pm}:\Ga(\S^+\otimes W^\pm)\to \Ga(\S^-\otimes W^\pm)$ be the twisted $\spinc$ Dirac operator, and define
\begin{equation}\label{TDirac}
\D_W:=
\begin{bmatrix}
0 & \D_{W^-}^* \\[3pt]
\D_{W^+} & 0
\end{bmatrix}:\Ga(E_W^+)\to\Ga(E_W^-).
\end{equation}
Given a complex-linear bundle map $\psi: W^+\to W^-$, define
\begin{equation}\label{TPer}
\A_\psi:=\begin{bmatrix} \ga\otimes\psi & 0 \\[3pt] 0 & \ga\otimes\psi^* \end{bmatrix}:E_W^+\to E_W^-.
\end{equation}
The principal symbols of $\D_W$ and $D_W^*$ are the two restrictions of Clifford multiplication.
Hence the concentration condition \eqref{CC} for $\A_\psi$ becomes
\begin{equation}\label{OurCC}
\c_{E_W}(\xi)\circ\A_\psi+\A_\psi^*\circ\c_{E_W}(\xi)=0\quad \forall \xi\in T^*M.
\end{equation}
Since $\ga$ anticommutes with $\c(\xi)$ for any $\xi\in T^*M$,
the pair $(\D_W,\A_\psi)$ satisfies the concentration condition.

\begin{rem}
Proposition~2.7 of \cite{PR} shows that each complex perturbation of $\D_W$ that satisfies the concentration condition~\eqref{CC}
has the form~\eqref{TPer} for some complex-linear bundle map $W^+\to W^-$.
For any real-linear bundle map $\psi:W^+\to W^-$, the associated real perturbation $\A_\psi$ also satisfies \eqref{CC}.
However, not every real perturbation satisfying \eqref{CC} is of the form $\A_\psi$; see, for example, \cite{Nag, L}.
\end{rem}

\subsection{The twisting bundle and the bundle map}
\label{TBandBM}

We now carry out the main construction of the paper. The twisting bundles are chosen so that the associated Dirac operator
satisfies \eqref{indBMY}, while the bundle map is adapted to the given $J$-holomorphic curve $C$
so that its singular set and kernel data have the required geometry
for index localization.

\subsubsection{The twisting bundle}

Henceforth, we fix the twisting bundle $W=W^+\oplus W^-$ by setting
$$
W^+:=\L^{1,0}M\oplus\ul{\C}
\qquad\text{and}\qquad
W^-:=\sl(\L^{1,0}M).
$$
Here $\sl(\L^{1,0}M)$ is the bundle of trace-free complex-linear endomorphisms
of $\L^{1,0}M$ and is the complexification of the adjoint bundle $\su(\L^{1,0}M)$.
Let $E:=E_{W}$ and  $\D:=\D_{W}$ denote the twisted bundle and Dirac operator defined by \eqref{Tbdle} and \eqref{TDirac}, respectively.

\begin{lemma}
$\ind\D=3c_2(M)-c_1^2(M)$.
\end{lemma}

\begin{proof}
Let $V=\L^{1,0}M$. Since the pairing
$V\otimes V\to \det V:v\otimes w\mapsto v\w w$ is nondegenerate, we have
$V\cong \Hom(V,\det(V))\cong V^*\otimes \det(V)$, hence
$V^*\cong V\otimes \det(V^*)$. Since $\rank V=2$, it follows that
\begin{align*}
\End(V)&\cong V\otimes V^*\cong V\otimes V\otimes \det(V^*)\cong \big(\Sym^2(V)\oplus \L^2 V\big)\otimes\det(V^*)
\\[2pt]
&\cong \Sym^2(V)\otimes\det(V^*)\oplus \ul{\C}.
\end{align*}
Under this decomposition, the trivial summand corresponds to scalar endomorphisms. Thus,
$$
W^-\cong \Sym^2(V)\otimes\det(V^*)\cong \Sym^2(\L^{1,0}M)\otimes \L^{0,2}M.
$$
The lemma now follows from the Atiyah--Singer index theorem:
\begin{align*}
\ind\D_{W}&=\ind\D_{W^+}-\ind\D_{W^-}=\fdfrac14\big(c_1^2(M)-3c_2(M)\big)-\fdfrac54\big(c_1^2(M)-3c_2(M)\big)
\\[3pt]
&=3c_2(M)-c_1^2(M).
\end{align*}
\end{proof}

\subsubsection{The complex-linear bundle map}

We next use the given curve $C$ to construct a bundle map $\psi:W^+\to W^-$ that will determine a perturbation~\eqref{TPer}.
For any $\a\in\Om^{1,0}(M)$ and $\eta\in\Om^{2,0}(M)$, define
$$
\a_\eta:=\iota(\a^\sharp)\eta\in\Om^{1,0}(M),
$$
where $\a^\sharp\in \Ga(T^{1,0}M)$ denotes the metric dual of $\a$ with respect to the induced hermitian metric.

The following lemma is a key observation in our construction.

\begin{lemma}\label{useful}
The sections $\a$, $\eta$, and $\a_\eta$ satisfy
$$
\lb\a,\a_\eta\rb=0\qquad\text{and}\qquad
|\a_\eta|=|\a||\eta|.
$$
\end{lemma}

\begin{proof}
In a unitary frame $\{ \mu_1,\mu_2\}$ for $T^{1,0}M$ and its dual coframe $\{\t^1,\t^2\}$ for $\L^{1,0}M$,
if $\a=\a_1\t^1+\a_2\t^2$ and $\eta=\eta_{12}\t^1\w\t^2$, then
$\a^\sharp =\ov{\a}_1\mu_1+\ov{\a}_2\mu_2$ and $\a_\eta = \eta_{12}\big(-\ov{\a}_2\t^1+\ov{\a}_1\t^2\big)$.
The lemma follows immediately from these formulas.
\end{proof}

Throughout the rest of the paper, we fix a transverse section $\eta$ of the canonical bundle $K_M=\L^{2,0} M$ whose zero set is $C$.
We also fix a transverse section $\a$ of $\L^{1,0}M$.


\begin{defn}\label{D:bdle-map}
Define $\psi:W^+\to W^-$ by
\begin{equation}\label{bdle-map}
\psi(\b,f)=\a_\eta^\sharp\otimes \b -\fdfrac12 \lb \b,\a_\eta\rb Id + f \a^\sharp\otimes \a_\eta,
\end{equation}
and let $\A:=\A_{\psi}:E^+\to E^-$ be the complex-linear perturbation defined by~\eqref{TPer}.
\end{defn}

The identity $\lb\a,\a_\eta\rb=0$ implies that $\psi(\b,f)$ is trace-free.
With respect to the induced hermitian metrics, the adjoint of $\psi$ is given by
\begin{equation}\label{adjoint}
\psi^*(B)=\pig(B\a_\eta,\lb B\a,\a_\eta\rb\pig)\quad\text{for }B\in W^-=\sl(\L^{1,0}M).
\end{equation}

\subsubsection{The singular set and kernel bundles}

Let $Z_\a:=\a^{-1}(0)$. Since $\a$ is transverse, $Z_\a$ is a finite set.
The lemma below proves (F1) in the Introduction.

\begin{lemma}\label{ss-ker}
The map $\psi$ is an isomorphism over $M\setminus (Z_\a\cup C)$ and vanishes on $Z_\a\cup C$. Hence $Z_{\A}=Z_\a\cup C$, and
$\A|_p=0$ for every $p\in Z_{\A}$. In particular,
$$
\ker\big(\A|_{Z_{\A}}\big)=E^+|_{Z_{\A}}\qquad\text{and}\qquad\ker\big(\A^*|_{Z_{\A}}\big)=E^-|_{Z_{\A}}.
$$
\end{lemma}

\begin{proof}
Let $q\in M\setminus (Z_\a\cup C)$ and suppose $\psi^*(B)=0$ for some $B\in W^-|_q$.
Then $B\a_\eta=0$ and $\lb B\a,\a_\eta\rb=0$ by \eqref{adjoint}. Since $B$ is trace-free and, by Lemma~\ref{useful},
$\{\a, \a_\eta\}$ is an orthogonal basis of $\L^{1,0}M|_q$,
$$
0=\tr B=|\a|^{-2} \lb B\a,\a\rb + |\a_\eta|^{-2}\lb B\a_\eta,\a_\eta\rb.
$$
We thus have both $\lb B\a,\a\rb=0$ and $\lb B\a,\a_\eta\rb=0$, so $B\a=0$.
Together with $B\a_\eta=0$, we conclude $B=0$. Therefore $\psi^*$ is injective at $q$.
The equality of the ranks of $W^+$ and $W^-$ then implies that both $\psi^*$ and $\psi$ are isomorphisms at $q$.
Because $\psi|_p=0$ for every $p\in Z_\a\cup C$ by \eqref{bdle-map} and the chirality operator is an isomorphism,
it follows from \eqref{TPer} that $Z_{\A}=Z_\a\cup C$ and  $\A|_p=0$ for every $p\in Z_{\A}$.
\end{proof}

Since the pair $(\D,\A)$ satisfies the concentration condition \eqref{CC}, the kernels of $\D+s\A$ and $\D^*+s\A^*$
concentrate near the singular set $Z_\a\cup C$ as $s\to\infty$.

\begin{rem}
The grading of $W$ is chosen so that $\ind\D$ is given by \eqref{indBMY}. With this grading,
$\det(W^-)^*\otimes \det(W^+) \cong K_M$, so $\det\psi^*$ is a section of $K_M$.
The construction of $\psi$ using the $J$-holomorphic curve $C$ is motivated by Example~6 of \cite{M1}.
\end{rem}

\section{Index localization theorem}
\label{ILT}

In this section, we temporarily return to the general case considered in Section~\ref{CPair} and review the index localization theorem of \cite{M2}.
Set
$$
E:=E_W,\qquad \c_E:=\c_{E_W},\qquad\D:=\D_W,\qquad\A:=\A_\psi,
$$
where $E_W$, $\c_{E_W}$, $\D_W$, and $\A_\psi$ are defined in Section~\ref{CPair}.
Let $Z$ be a component of the singular set $Z_\A$ satisfying Condition~(C1) below, and let $N$ be its normal bundle.
Set $r:=\codim Z$. For a local orthonormal frame $\{e_1,\cdots,e_r\}$ for $N$ with its dual coframe $\{e^j\}$, set
\begin{equation}\label{nder}
\A_j:=(\nabla_{e_j}\A)|_{\ker(\A|_Z)}:\ker(\A|_Z)\to \ker(\A^*|_Z),\qquad
\A^*_j:=(\nabla_{e_j}\A^*)|_{\ker(\A^*|_Z)}=(\A_j)^*,
\end{equation}
and define
\begin{equation}\label{CCoperator}
\begin{array}{ll}
{\ds M_j^+:=-\c_E(e^j)\A_j,}\qquad &{\ds \CC^+_Z:=\sum_{j=1}^r M_j^+:\ker(\A|_Z)\to \ker(\A|_Z),}
\\
{\ds M_j^-:=-\c_E(e^j)\A^*_j,}\qquad &{\ds \CC^-_Z:=\sum_{j=1}^r M_j^-:\ker(\A^*|_Z)\to \ker(\A^*|_Z).}
\end{array}
\end{equation}
The endomorphisms $\CC^\pm_Z$ are independent of the choice of local orthonormal frame for $N$, and hence are globally defined on $Z$.

The hypotheses for index localization are the following conditions on $Z$.
\begin{itemize}
\item[(C1)]
$Z$ is a compact submanifold of $M$, $\ker(\A|_Z)$ and $\ker(\A^*|_Z)$ are vector bundles over $Z$,  and
for $v\in TM|_Z$,
\begin{equation}\label{P4}
\nabla_v\A\big(\ker(\A|_Z)\big)\subset  \ker(\A^*|_Z)\qquad\text{and}\qquad \nabla_v\A^*\big(\ker(\A^*|_Z)\big)\subset \ker(\A|_Z).
\end{equation}

\item[(C2)]
There is a positive-definite hermitian endomorphism $Q_Z:\ker(\A|_Z)\to \ker(\A|_Z)$ such that
$$
\A_j^*\A_k + \A_k^*\A_j=2\delta_{jk}Q_Z\quad(1\leq j,k\leq r).
$$
\item[(C3)]
$Q_Z$ has eigenvalues $\la_1^2,\cdots,\la_d^2$, where $\la_1,\cdots,\la_d:Z\to (0,\infty)$ are smooth and their graphs are pairwise disjoint.
\end{itemize}

We use the following form of the index localization theorem.

\begin{theorem}[\cite{M2}]\label{M2ILT}
Assume that every component $Z$ of $Z_\A$ satisfies Conditions~(C1)--(C3).
Then the spectrum of  $\CC_Z^\pm$ is  $\{(r-2k)\la_\ell:1\leq\ell\leq d, 0\leq k\leq r\}$ if $\dim Z>0$,
and is contained in this set if $\dim Z=0$.
The corresponding eigenspaces form vector bundles over $Z$.
Set
$$
S_\ell^\pm:=\ker\big(\CC_Z^\pm-r\la_\ell Id\big),
$$
and let $P_\ell^\pm:E^\pm|_Z\to S^\pm_\ell$ be the orthogonal projection. Set
$$
S_Z^\pm:=\bigoplus_{\ell=1}^d S_\ell^\pm\qquad\text{and}\qquad P^\pm:=\sum_{\ell=1}^d P_\ell^\pm.
$$
The Clifford multiplication $\c_E$ restricts to $T^*Z\otimes S_Z^\pm\to S_Z^\mp$, and the connections $P^\pm\circ\nabla|_{TZ}$
are compatible with this Clifford multiplication. Thus, the composition
$$
\D_Z: \Ga(S_Z^+)\xr{P^+\circ \nabla|_{TZ}} \Ga(T^*Z\otimes S_Z^+)\xr{\ \c_{E}\ } \Ga(S_Z^-)
$$
is a Dirac operator on $Z$. The index of $\D$ is then given by
$$
\ind\D=\sum_{Z}\ind\D_Z,
$$
where the sum ranges over the components of $Z_\A$.
\end{theorem}

\begin{rem}\label{Comparison}
Our notation and formulation differ from those of \cite{M2}. In \cite{M2}, the indices $i=0,1$ correspond to our superscripts $+$ and $-$, respectively,
while the sign $\pm$ in $S_\ell^{i\pm}$ records the sign of the eigenvalue $\pm r\la_\ell$. Since we consider only
the limit $s\to +\infty$ for $\D+s\A$, the bundles $S_\ell^{0+}$ and $S_\ell^{1+}$ in Definition~2.6 of \cite{M2} correspond to our $S_\ell^+$ and $S_\ell^-$, respectively.

By Proposition~B.1\,(2) of \cite{M2}, the endomorphisms $C^0$ and $C^1$ defined in equation~(2.11) of \cite{M2}  coincide with $\CC_Z^+$ and $\CC_Z^-$, respectively.
In our formulation, Condition~(C1) replaces the transversality assumptions~(1) and~(2). By Lemma~3.5\,(4) of \cite{M2}, Condition~(C2) is equivalent to
the non-degeneracy assumption~(3), while Condition~(C3) is precisely the stable degeneration assumption~(4).
Finally, with our sign convention for $M_j^\pm$, the eigenbundles entering the localization theorem are the $r\la_\ell$-eigenbundles of $\CC_Z^\pm$.
\end{rem}

By the preceding remark, the non-degeneracy hypothesis of \cite{M2} is equivalent to Condition~(C2).
Hence the following lemma shows that, at each isolated singular point $p$ with $\psi|_p=0$, this hypothesis imposes an even-rank condition on $W^\pm$.

\begin{lemma}\label{Obstruction}
Suppose that $\psi|_p=0$ and $\A=\A_\psi$ satisfies Condition~(C2) for $Z=p$.
Then $W^+$ and $W^-$ have the same even complex rank.
\end{lemma}

\begin{proof}
We will show that $W_p^\pm$ have the same even complex dimension. Since
$$
\A_j=\begin{bmatrix} \ga\otimes\nabla_j\psi & 0 \\[3pt] 0 & \ga\otimes \nabla_j\psi^* \end{bmatrix}:(\S\otimes W)^+|_p \to (\S\otimes W)^-|_p,\qquad
\A_j^*=(\A_j)^*,
$$
and $\ga^*\ga=\ga^2=Id$, Condition~(C2) becomes $\Psi_j^*\Psi_k+\Psi_k^*\Psi_j=2\d_{jk}\wh{Q}$ ($1\leq j,k\leq 4$), where
$$
\Psi_j=\begin{bmatrix} \nabla_j\psi & 0  \\[3pt] 0 &\nabla_j\psi^*  \end{bmatrix}:W^+_p\oplus W^-_p \to W^-_p\oplus W^+_p,
\qquad \Psi_j^*=(\Psi_j)^*,
$$
and $\wh{Q}=(\Psi_j)^*\Psi_j:W_p\to W_p$ is a positive-definite hermitian endomorphism.
Note that
$$
\wt \Psi_j:=H\Psi_j=\Psi_j^*H= \begin{bmatrix} 0 & \nabla_j\psi^*   \\[3pt] \nabla_j\psi & 0   \end{bmatrix}:W_p\to W_p
\qquad\left( H:=\begin{bmatrix} 0 & Id \\[3pt] Id & 0\end{bmatrix}\right)
$$
is a self-adjoint odd isomorphism of $W_p=W_p^+\oplus W_p^-$ and $\wt \Psi_j\wt \Psi_k+\wt \Psi_k\wt \Psi_j=2\d_{jk}\wh{Q}$ ($1\leq j,k\leq 4$).
Since each $\wt \Psi_j$ commutes with $\wh{Q}$, on each $\la$-eigenspace $F_\la=F_\la^+\oplus F_\la^-$ of $\wh{Q}$, the odd isomorphisms
$$
\la^{-\frac12}\wt \Psi_j|_{F_\la}:F_\la\to F_\la\quad(1\leq j\leq 4)
$$
define a $\CCl_4$-module structure, where $\CCl_4:=\Cl_4\otimes_\R\C\cong M_4(\C)$.
The $\CCl_4$-module structure shows that $\dim_\C F_\la$ is divisible by $4$, and thus the odd isomorphism $\wt\Psi_j|_{F_\la}$ shows that
$\dim_\C F_\la^+=\dim_\C F_\la^-$ is even. Therefore, the assertion follows from $W_p^\pm=\sum_\la F_\la^\pm$, where
the sum is over the eigenvalues $\la$ of $\wh{Q}$.
\end{proof}

We end this Section by describing the case in which $Z$ is a point.

\begin{cor}\label{ILTPoint}
With the notation above, if $Z=p$ is a point, then its contribution to $\ind\D$ is
$$
\ind_p:=\ind\D_p=\dim\left[\bigcap_j\Big(\bigoplus_\ell\ker(M_j^+ -\la_\ell Id)\Big)\right]
- \dim\left[\bigcap_j\Big(\bigoplus_\ell\ker(M_j^- -\la_\ell Id)\Big)\right].
$$
\end{cor}

\begin{proof}

By differentiating the concentration condition \eqref{OurCC} and using $\c_E(e^j)\c_E(e^j)=-Id$,
we have $\c_E(e^j)\A_k=-\A_k^*\c_E(e^j)$ and $\A_k\c_E(e^j)=-\c_E(e^j)\A^*_k$. Together with Condition~(C2),
these show that $M_j^\pm$ are commuting hermitian endomorphisms with $(M_j^+)^2=Q_p$ and $(M_j^-)^2=-\c_E(e^j)Q_p\c_E(e^j)$.
Condition~(C3) therefore implies that the spectrum of each $M_j^\pm$  is contained in $\{\pm\la_\ell\}$.
Since the induced operator $\D_p$ is trivial, $\ind\D_p = \dim S_p^+ - \dim S_p^-$.
The corollary now follows from simultaneous diagonalization of the $M_j^\pm$.
\end{proof}

This corollary agrees with the formula in Theorem~5.4 of \cite{PR}, where the opposite sign convention is used in the definition of $M_j^\pm$.

\section{Localization along $C$}

We now return to the pair $(\D,\A)$ constructed in Section~\ref{TBandBM}.
Throughout the remaining sections, we assume $Z_\a\cap C=\emptyset$, i.e., $\a$ is nonvanishing on $C$.
In this section, we apply Theorem~\ref{M2ILT} to prove (F2) stated in the Introduction.

\subsection{Eigenvalues and eigenbundles}
\label{Ebdle}

Let $Z$ be a component of $C$ and let $N$ be its normal bundle.
Locally along $Z$, choose an orthonormal frame $\{e_1,e_2=Je_1,e_3,e_4=Je_3\}$ for $TM$ and its dual coframe
$\{e^i\}$ such that $\{e_1|_Z,e_2|_Z\}$ is an orthonormal frame for $N$.
Let $\{\t^1,\t^2\}$ be a local unitary frame for $\L^{1,0}M$ given by
\begin{equation*}
\t^1=\fdfrac{1}{\sqrt{2}}(e^1-iJe^1)=\fdfrac{1}{\sqrt{2}}(e^1+ie^2)
\qquad\text{and}\qquad\t^2=\fdfrac{1}{\sqrt{2}}(e^3-iJe^3)=\fdfrac{1}{\sqrt{2}}(e^3+ie^4).
\end{equation*}
Let $\{\mu_1,\mu_2\}$ denote the unitary frame for $T^{1,0}M$ dual to $\{\t^1,\t^2\}$.

\subsubsection{Clifford multiplication}

Recall that with respect to the chosen frames, Clifford multiplication on the canonical $\spinc$ spinor bundle $\S=\S^+\oplus\S^-$ is given by
$$
\c(\ov{\t}^k)=\sqrt{2}\,\ov{\t}^k\!\w \qquad \text{and}\qquad \c(\t^k)=-\sqrt{2}\,\iota(\ov{\mu}_k)\qquad(k=1,2),
$$
and hence
\begin{equation*}
\c(e^{2k-1}) =\vt^k\!\w\, -\iota(\vm_k)\qquad\text{and}\qquad\c(e^{2k})=i\big(\vt^k\!\w\, +\iota(\vm_k)\big)\qquad(k=1,2).
\end{equation*}
For ease of reference both here and in the later computation of the contribution from $Z_\a$,
we record the resulting identities in the following table.
\begin{equation}\label{CliffM}
\begin{array}{l|l|l|l}
\c(e^1)1=\vt^1 & \c(e^1)\vt^1=-1 &\c(e^2)1=i\vt^1 & \c(e^2)\vt^1=i
\\[3pt]
\c(e^1)\vt^1\w\vt^2=-\vt^2 &\c(e^1)\vt^2=\vt^1\w\vt^2 &\c(e^2)\vt^1\w\vt^2=i\vt^2   &\c(e^2)\vt^2=i\vt^1\w\vt^2
\\[3pt] \hline & & & \\[-6pt]
\c(e^3)1=\vt^2 & \c(e^3)\vt^1=-\vt^1\w\vt^2  &\c(e^4)1=i\vt^2 & \c(e^4)\vt^1=-i\vt^1\w\vt^2
\\[3pt]
\c(e^3)\vt^1\w\vt^2=\vt^1 &\c(e^3)\vt^2=-1 & \c(e^4)\vt^1\w\vt^2=-i\vt^1  &\c(e^4)\vt^2=i
\end{array}
\end{equation}
The table is organized by the two pairs $(e^{2k-1},e^{2k})$ corresponding to $\t^k$ ($k=1,2$), displayed in the upper and
lower halves, respectively, and within each half according to the chirality of the input spinors.
This arrangement allows the relevant eigensections along $Z$ and eigenspaces at each point of $Z_\a$ to be read off directly.

\subsubsection{Orientation preserving or reversing}

Since $\eta$ is a transverse section of $K_M$, its normal
derivative $\nabla\eta:N\to K_M|_Z$ is a real bundle isomorphism.
By Lemma~5.9 of \cite{M1}, after deforming $\eta$ without changing its zero set,
we may assume that $\nabla\eta$ is orthogonal.
Lemma~\ref{ss-ker} then implies that Condition~(C1) in Section~\ref{ILT} holds for $Z$.

Since $Z$ is connected, $\nabla\eta$ is either orientation preserving or orientation reversing with respect to the complex orientations.
Set
$$
\ep := \left\{
\begin{array}{rl}
1 & \text{if $\nabla\eta$ is orientation preserving},\\[3pt]
-1 &\text{if $\nabla\eta$ is orientation reversing}.
\end{array}\right.
$$
By the adjunction formula,
\begin{equation*}
N\,\xr{\nabla\eta}\,K_M|_Z\,\cong \ov{N}\otimes T^*Z.
\end{equation*}
When $\ep=1$, $\nabla\eta$ is complex-linear, and this isomorphism implies that $N$ (with its induced holomorphic structure)
is a theta characteristic. When $\ep=-1$, $\nabla\eta$ is a conjugate-linear isomorphism, so it
induces a complex-linear isomorphism $\ov{N}\cong\ov{N}\otimes T^*Z$.
Hence $T^*Z$ is trivial, and thus the genus of $Z$ is $g(Z)=1$.
Since the components of $C$ are pairwise disjoint and their union represents the canonical class,
$c_1(K_M)[Z]=C\cdot Z=Z\cdot Z=c_1(N)[Z]$, where the last equality follows because
the self-intersection number of $Z$ in $M$ is the degree of its normal bundle.
The adjunction formula then gives
$$
2c_1(N)[Z]=c_1(N)[Z]+c_1(K_M)[Z]=2g(Z)-2=0.
$$
Therefore
\begin{equation}\label{Torus}
g(Z)=1\qquad\text{and}\qquad c_1(N)=0.
\end{equation}

\subsubsection{The non-degeneracy condition}
\label{non-degeneracy}

For $j=1,2$, write $\nabla_j:=\nabla_{e_j}$.
The following identities are useful for our subsequent discussion.
\begin{equation}\label{orientation}
|\nabla_1\eta|=1,\qquad\nabla_2\eta=\ep i\nabla_1\eta,\qquad
\a_{\nabla_2\eta}=\ep i\a_{\nabla_1\eta},\qquad(\a_{\nabla_2\eta})^\sharp=-\ep i(\a_{\nabla_1\eta})^\sharp.
\end{equation}

Using \eqref{bdle-map}, \eqref{adjoint}, and the fact that  $\eta|_Z=0$, we obtain, for each $1\leq j\leq 2$,
\begin{align*}
\psi_j(\b,f)&:=(\nabla_j\psi)|_Z(\b,f)=(\a_{\nabla_j\eta})^\sharp\otimes \b -\fdfrac12\lb\b, \a_{\nabla_j\eta}\rb Id
+ f\a^\sharp\otimes\a_{\nabla_j\eta},
\\[3pt]
\psi_j^*(B) &:= (\nabla_j\psi^*)|_Z B=\pig(B\a_{\nabla_j\eta},\lb B\a,\a_{\nabla_j\eta}\rb\pig).
\end{align*}
Along $Z$, $|\nabla_k\eta|=1$ by \eqref{orientation}. As $\lb\a,\a_{\nabla_k\eta}\rb=0$ by Lemma~\ref{useful}, it follows that:
\begin{itemize}
\item[(a)]
Each $\psi_j$, and hence each $\psi_j^*$, is an isomorphism over $Z$ by the same argument as in the proof of Lemma~\ref{ss-ker}.
\item[(b)]
$\psi_k^*\psi_j(\b,f)=\Big(\lb \a_{\nabla_k\eta},\a_{\nabla_j\eta}\rb\b-\fdfrac12\lb\b, \a_{\nabla_j\eta}\rb\a_{\nabla_k\eta}\,,
f|\a|^2\lb \a_{\nabla_j \eta},\a_{\nabla_k\eta}\rb\Big)$.
\item[(c)]
$\psi_j\psi_k^*(B)=(\a_{\nabla_j\eta})^\sharp\otimes B\a_{\nabla_k\eta} -\fdfrac12\lb B\a_{\nabla_k\eta}, \a_{\nabla_j\eta}\rb Id
+ \lb B\a,\a_{\nabla_k\eta}\rb\a^\sharp\otimes\a_{\nabla_j\eta}$.
\end{itemize}

Recalling that $\ker(\A|_Z)=E^+|_Z$ and $\ker(\A^*|_Z)=E^-|_Z$ by Lemma~\ref{ss-ker},
let $\A_j$ ($j=1,2$) denote the normal derivatives defined in Section~\ref{ILT}.
Using (b), (c), \eqref{orientation}, and
\begin{equation}\label{nder-C}
\A_j=\begin{bmatrix}
\ga\otimes\psi_j & 0 \\[3pt]
0 & \ga\otimes\psi_j^*
\end{bmatrix}:E^+|_Z\to E^-|_Z,\qquad \A_j^*=(\A_j)^*,
\end{equation}
we obtain $Q=\A_1^*\A_1=\A_2^*\A_2$ and $\A_1^*\A_2+\A_2^*\A_1=0$,
where $Q$ is positive definite over $Z$ by (a).
Thus Condition~(C2) in Section~\ref{ILT} holds for $Z$.

\subsubsection{Eigenvalues}\label{Eigenvalues}

It remains to verify that Condition~(C3) holds.
Since $C\cap Z_\a=\emptyset$, after a deformation of $\a$ supported in a neighborhood of $C$ disjoint from $Z_\a$,
we may assume that $|\a|>1$ along $C$. By Lemma~\ref{useful} and \eqref{orientation},  $|\a_{\nabla_1\eta}|=|\a|$ and
$\{\a_{\nabla_1\eta},\a\}$ is a local orthogonal frame for $\L^{1,0}M$. Noting that
$$
Id =\fdfrac{1}{|\a|^2}\a^\sharp\otimes\a + \fdfrac{1}{|\a|^2}(\a_{\nabla_1\eta})^\sharp\otimes\a_{\nabla_1\eta},
$$
we set
$$
\begin{array}{l|l|l}
\a_1:=(\a_{\nabla_1\eta},0) & \a_2:=(\a,0) & \a_3:=(0,1)
\\[3pt] \hline & & \\[-5pt]
B_1:=-\fdfrac12\big(\a^\sharp\otimes\a-(\a_{\nabla_1\eta})^\sharp\otimes \a_{\nabla_1\eta}\big) &
B_2:=(\a_{\nabla_1\eta})^\sharp\otimes\a & B_3:=\a^\sharp\otimes\a_{\nabla_1\eta}
\end{array}
$$
It follows that
\begin{itemize}
\item[(i)]
$\{\a_1,\a_2,\a_3\}$ and $\{B_1,B_2,B_3\}$ are local orthogonal frames for $W^+|_Z$ and $W^-|_Z$, respectively.
\item[(ii)]
$\psi_1(\a_\ell)=B_\ell$ ($1\leq \ell\leq 3$), $\psi_1^*(B_1)=\fdfrac12|\a|^2\a_1$, $\psi_1^*(B_2)=|\a|^2\a_2$, $\psi_1^*(B_3)=|\a|^4\a_3$.
\end{itemize}
Set
$$
\la_1:=\fdfrac{1}{\sqrt{2}}|\a|,\qquad \la_2:=|\a|,\qquad \la_3:=|\a|^2.
$$
By (i) and (ii), the eigenvalues of the endomorphism
$$
Q=\A_1^*\A_1=\begin{bmatrix}
Id\otimes\psi_1^*\psi_1&0\\[3pt]
0&Id\otimes\psi_1\psi_1^*
\end{bmatrix}
$$
of $E^+|_Z$ are $\la_1^2, \la_2^2, \la_3^2$ because the local sections
$$
\big(1\otimes \a_\ell,0\big),\quad \big((\vt^1\w\vt^2)\otimes \a_\ell,0\big),\quad \big(0,\vt^1\otimes B_\ell\big),\quad\big(0,\vt^2\otimes B_\ell\big)
\qquad(1\leq \ell\leq 3)
$$
span the $\la_\ell^2$-eigenspace in each fiber, and these together form a local orthogonal frame for $E^+|_Z$.
Since $|\a|>1$ along $Z$, we have $\la_1<\la_2<\la_3$ everywhere on $Z$.
Thus the graphs of the functions $\la_\ell:Z\to (0,\infty)$ ($1\leq \ell\leq 3$) are pairwise disjoint.
Condition~(C3) in Section~\ref{ILT} holds for $Z$.

\subsubsection{Eigenbundles}

Since $\ga|_{\S^\pm}=\pm Id$, the endomorphisms $\CC_Z^\pm:E^\pm|_Z\to E^\pm|_Z$ defined by \eqref{CCoperator} become
$$
\CC_Z^+=\begin{bmatrix}
0 & \sum\c(e^j)\otimes\psi_j^*  \\[3pt] -\sum\c(e^j)\otimes\psi_j & 0
\end{bmatrix},\quad
\CC_Z^-=\begin{bmatrix}
0 & \sum\c(e^j)\otimes\psi_j  \\[3pt] -\sum\c(e^j)\otimes\psi_j^* & 0
\end{bmatrix},
$$
where the sums are over $j=1,2$.
By Theorem~\ref{M2ILT}, each eigenspace of $\CC^\pm_Z$ forms a vector bundle over $Z$.

Let $\ep_\ell=\ep$ for $\ell=1,2$ and $\ep_\ell=-\ep$ for $\ell=3$. By \eqref{orientation} and Section~\ref{Eigenvalues}\,(ii), we have
$$
\psi_2(\a_\ell)=-\ep_\ell i \psi_1(\a_\ell)=-\ep_\ell iB_\ell\qquad\text{and}\qquad \psi_2^*(B_\ell)=\ep_\ell i\psi_1^*(B_\ell)=\ep_\ell i\la_\ell^2\a_\ell.
$$
It then follows from the Clifford multiplication \eqref{CliffM} that
\begin{align*}
\CC_Z^+\begin{bmatrix}1\otimes \la_\ell \a_\ell \\[3pt] \vt^1\otimes B_\ell \end{bmatrix}
&=-(1+\ep_\ell)\la_\ell\begin{bmatrix}1\otimes \la_\ell \a_\ell \\[3pt] \vt^1\otimes B_\ell \end{bmatrix},
&\CC_Z^+\begin{bmatrix}\vt^1\w\vt^2\otimes \la_\ell\a_\ell\\[3pt] \vt^2\otimes B_\ell\end{bmatrix}
&=(1-\ep_\ell)\la_\ell\begin{bmatrix}\vt^1\w\vt^2\otimes \la_\ell\a_\ell\\[3pt] \vt^2\otimes B_\ell\end{bmatrix},
\\[3pt]
\CC_Z^-\begin{bmatrix}1\otimes B_\ell \\[3pt] \vt^1\otimes \la_\ell\a_\ell \end{bmatrix}
&=-(1-\ep_\ell)\la_\ell\begin{bmatrix}1\otimes B_\ell \\[3pt] \vt^1\otimes \la_\ell\a_\ell \end{bmatrix},
&\CC_Z^-\begin{bmatrix}\vt^1\w\vt^2\otimes B_\ell\\[3pt] \vt^2\otimes \la_\ell\a_\ell\end{bmatrix}
&=(1+\ep_\ell)\la_\ell\begin{bmatrix}\vt^1\w\vt^2\otimes B_\ell\\[3pt] \vt^2\otimes \la_\ell\a_\ell\end{bmatrix}.
\end{align*}
From these, one sees that the $0$-eigenbundles of $\CC_Z^\pm$ both
have rank $6$ and that each $\pm2\la_\ell$-eigenbundle is a line bundle.
The $2\la_\ell$-eigenline bundles $S_\ell^\pm$ appearing in Theorem~\ref{M2ILT} have the following local frames:
\begin{equation}\label{localframe}
\begin{array}{c|c|c|c}
S^+_\ell (\ep_\ell=1) & S^+_\ell (\ep_\ell=-1) & S^-_\ell (\ep_\ell=1) & S^-_\ell (\ep_\ell=-1)
\\[3pt]
\hline & & &  \\[-5pt]
\begin{bmatrix}1\otimes \la_\ell \a_\ell \\[3pt] -\vt^1\otimes B_\ell \end{bmatrix}
&\begin{bmatrix}\vt^1\w\vt^2\otimes \la_\ell\a_\ell\\[3pt] \vt^2\otimes B_\ell\end{bmatrix}
&\begin{bmatrix}\vt^1\w\vt^2\otimes B_\ell\\[3pt] \vt^2\otimes \la_\ell\a_\ell\end{bmatrix}
&\begin{bmatrix}1\otimes B_\ell \\[3pt] -\vt^1\otimes \la_\ell\a_\ell \end{bmatrix}
\end{array}
\end{equation}

\subsection{The induced Dirac operator on $Z$}

Let $S_Z^\pm=\bigoplus_\ell S_\ell^\pm$ and let $\D_Z:\Ga(S_Z^+)\to \Ga(S_Z^-)$ be the induced Dirac operator on $Z$ given in Theorem~\ref{M2ILT}.
It remains to compute $\ind\D_Z$, the contribution from $Z$ to $\ind\D$. We do this by applying the
Atiyah--Singer index theorem, after determining the isomorphism classes of the line bundles $S^\pm_\ell$
from their transition functions.

\subsubsection{Transition functions of the eigenbundles}

The local frames of the eigenline bundles $S_\ell^\pm$ in \eqref{localframe} depend on
the choice of the adapted orthonormal frame $\{e_i\}$ along $Z$, introduced at the beginning of Section~\ref{Ebdle}.
We now compute how they transform under a change of adapted frame. The resulting scalar factors
are the transition functions of $S_\ell^\pm$, which determine  isomorphism classes of these eigenline bundles.

On an overlap, choose another
orthonormal frame $\{f_1,f_2=Jf_1,f_3,f_4=Jf_3\}$ for $TM$ and its dual coframe
$\{f^i\}$ such that $\{f_1|_Z,f_2|_Z\}$ is an orthonormal frame for $N$ and along $Z$,
$$
\begin{bmatrix} e_{2k-1} \\ e_{2k} \end{bmatrix} =
\begin{bmatrix} a_k & b_k \\ c_k  & d_k \end{bmatrix} \begin{bmatrix} f_{2k-1} \\ f_{2k} \end{bmatrix}
\quad(k=1,2).
$$
Since both frames are orthonormal and satisfy $Je_{2k-1}=e_{2k}$ and $Jf_{2k-1}=f_{2k}$, we have $a_k^2+b_k^2=1$, $c_k=-b_k$, and $d_k=a_k$.
It follows that
$$
\mu_k=\fdfrac{1}{\sqrt{2}}(e_{2k-1}-ie_{2k})=g_k\nu_k\qquad\text{where}\qquad
\nu_k:=\fdfrac{1}{\sqrt{2}}(f_{2k-1}-if_{2k}),\quad g_k:=a_k+ib_k.
$$
Here $\mu_1$ and $\nu_1$ are local unitary frames for the normal bundle $N$, and, as $TM|_Z=N\oplus TZ$,
$\mu_2$ and $\nu_2$ are local unitary frames for the tangent bundle $TZ$. Thus $g_1$ and $g_2$ are the transition functions of $U(1)$-bundles $N$ and $TZ$, respectively.
Observe that
\begin{equation}\label{tran1}
\vt^k=\fdfrac{1}{\sqrt{2}}(e^{2k-1}-ie^{2k})=g_k\vbt^k\qquad\text{where}\qquad\vbt^k:=\fdfrac{1}{\sqrt{2}}(f^{2k-1}-if^{2k}).
\end{equation}

Let $\a_\ell^\p$ and $B_\ell^\p$ ($1\leq \ell\leq  3$) denote the local sections introduced in Section~\ref{Eigenvalues}
and defined using the frame $\{f_1,f_2\}$. Noting $(g_1)^\ep=\ov{g}_1$ if $\ep=-1$, we have $\nabla_{e_k}\eta=(g_1)^\ep\nabla_{f_k}\eta$ ($k=1,2$). Hence
\begin{equation}\label{tran2}
\a_1=(g_1)^\ep\a_1^\p,\quad \a_2=\a_2^\p,\quad \a_3=\a_3^\p,\quad B_1=B_1^\p,\quad B_2=(g_1)^{-\ep} B_2^\p,\quad B_3=(g_1)^\ep B_3^\p.
\end{equation}

Using the local frames in \eqref{localframe}, together with \eqref{tran1} and \eqref{tran2},
we obtain the following transition functions of $S^\pm_\ell$ with respect to
the frames induced by $\{e_i\}$ and $\{f_i\}$.
\begin{equation}\label{tran3}
\begin{array}{c|cc|cc|cc}
& S^+_1 & S^-_1 & S^+_2 & S^-_2 & S^+_3 & S^-_3
\\[3pt]
\hline & & & & & & \\[-5pt]
\ep=1 & g_1 & g_1g_2 & 1 & g_2 & g_1g_2 & g_1
\\[3pt]
\hline & & & & & & \\[-5pt]
\ep=-1 & g_2 & 1 & g_1g_2 & g_1 & 1 & g_2
\end{array}
\end{equation}

\subsubsection{Index of $\D_Z$}

A direct calculation shows that the Clifford multiplication $\c_{E}:T^*Z\otimes S_Z^\pm\to S_Z^\mp$ further restricts to
$T^*Z\otimes S^\pm_\ell\to S^\mp_\ell$ for each $1\leq \ell\leq 3$; in fact, this also follows from Proposition~3.7\,(2) of \cite{M2}.
Hence the operator $\D_Z$ agrees with $\bigoplus_\ell \D_{Z,\ell}$ up to a zeroth-order term, where
$$
\D_{Z,\ell}:\Ga(S^+_\ell)\xr{P^+_\ell\circ\nabla|_{TZ}}\Ga(T^*Z\otimes S^+_\ell)\xr{c_{E}} \Ga(S^-_\ell).
$$

Let $\S_Z=\S_Z^+\oplus\S_Z^-$ denote the spinor bundle of the canonical
$\spinc$ structure on $Z$, where
$$
\S_Z^+=\L^{0,0}Z\cong \ul{\C}\qquad\text{and}\qquad \S_Z^-=\L^{0,1}Z.
$$
Since each $S_\ell:=S^+_\ell\oplus S^-_\ell$ is a $\Cl(T^*Z)$-module and $S^\pm_\ell$ are line bundles,
$S_\ell\cong \S_Z\otimes W_\ell$ as  Clifford modules for some complex line bundle $W_\ell$; see, for example, \cite[p.~619]{Ni}.
Here the superscripts $+$ and $-$ on $S_\ell^\pm$ refer to the grading inherited from $E^\pm|_Z$, which need not
agree with the chirality grading of $S_\ell$ as a $\Cl(T^*Z)$-module. The latter is the decomposition into $\pm 1$-eigenbundles
of the chirality operator $\ga_Z:=i\c_E(e^3)\c_E(e^4)$. Let
$$
\d_\ell:=\left\{\begin{array}{rl}
1 &\text{if }\ga_Z|_{S^+_\ell}=Id,\\[3pt]
-1 &\text{if }\ga_Z|_{S^+_\ell}=-Id.
\end{array}\right.
$$
Thus, $S_\ell^\pm\cong \S_Z^\pm\otimes W_\ell$ if $\d_\ell=1$,
whereas $S_\ell^\pm \cong \S_Z^\mp\otimes W_\ell$ if $\d_\ell=-1$.

\begin{lemma}\label{index-sign}
${\ds \ind\D_{Z,\ell}=\d_\ell\int_Z \ch(W_\ell)\td(Z)}$. 
\end{lemma}

\begin{proof}
Denote by $\D_{W_\ell}:\Om^0(W_\ell)\to\Om^{0,1}(W_\ell)$ the $W_\ell$-twisted $\spinc$ Dirac operator.
Under the $\Cl(T^*Z)$-module isomorphism $S_\ell\cong\S_Z\otimes W_\ell$, the Clifford multiplication on $S_\ell$
is identified with $\c\otimes Id_{W_\ell}$. Thus the symbol of $\D_{Z,\ell}$ is identified with that of
$\D_{W_\ell}$ if $\d_\ell=1$, and with that of $\D_{W_\ell}^*$ if $\d_\ell=-1$.
Therefore the lemma follows from the Atiyah--Singer index theorem.
\end{proof}

Since $g_1$ and $g_2$ are the transition functions of $N$ and $TZ$, respectively,
\eqref{tran3} gives the following isomorphisms of complex line bundles.
The signs $\d_\ell$ in the table below are obtained by the action of $\ga_Z$ on $S^+_\ell$, computed using \eqref{CliffM} and \eqref{localframe}.
\begin{equation}\label{tran4}
\begin{array}{r|ccc|ccc|ccc}
& \d_1& S^+_1 & S^-_1  & \d_2 & S^+_2 & S^-_2  & \d_3 & S^+_3 & S^-_3
\\[3pt]
\hline & & & & & & & & & \\[-5pt]
\ep=1
& +1 & N & N\otimes TZ & +1 & \ul{\C} & TZ & -1 & N\otimes TZ  & N
\\[3pt]
\hline & & & & & & & & & \\[-5pt]
\ep=-1
& -1 & TZ & \ul{\C}  & -1& N\otimes TZ & N & +1 & \ul{\C} & TZ
\end{array}
\end{equation}
%
%
Since $\S_Z^+\cong \ul{\C}$, we have $W_\ell\cong S_\ell^+$ if $\d_\ell=1$ and $W_\ell\cong S_\ell^-$ if $\d_\ell=-1$.
Hence \eqref{tran4} yields
$$
(W_1,W_2,W_3)\cong\left\{
\begin{array}{cl}
(N,\ul{\C},N) &\text{if }\ep=1, \\[3pt]
(\ul{\C},N,\ul{\C}) &\text{if }\ep=-1.
\end{array}\right.
$$
Combining Lemma~\ref{index-sign}, \eqref{Torus}, and this identification of $W_\ell$,  we obtain
\begin{equation}\label{ch-ind}
\ind \D_Z=\sum_{\ell=1}^3\ind \D_{Z,\ell}=\left\{
\begin{array}{ll}
{\ds \int_Z \ch(N+\ul{\C}-N)\td(Z)=1-g(Z)} &\text{when }\ep=1,
\\[8pt]
{\ds \int_Z \ch(-\ul{\C}-N+\ul{\C})\td(Z) = 0 = 1-g(Z) }&\text{when }\ep=-1.
\end{array}\right.
\end{equation}

The following proposition completes the proof of (F2) in the Introduction.

\begin{prop}\label{CalC}
The contribution from $C$ to $\ind\D$ is $-c_1^2(M)$.
\end{prop}

\begin{proof}
Write $C=\sqcup_k Z_k$, and let $N_k$ denote the normal bundle of
$Z_k$. Let $\D_{Z_k}$ denote the induced Dirac operator on $Z_k$.
It then follows from Theorem~\ref{M2ILT} and \eqref{ch-ind} that the contribution from $C$ is
$$
\sum_k \ind\D_{Z_k}=\sum_k (1 - g(Z_k)) =-c_1^2(M),
$$
where the last equality follows from the adjunction formula.
\end{proof}

\section{Localization at $Z_\a$}
\label{ContributionC0}

At each point $p$ of $Z_\a$, since $\psi|_p=0$ and $W^\pm$ have complex rank $3$,
Lemma~\ref{Obstruction} shows that we cannot directly apply Theorem~\ref{M2ILT}.
In this section, we therefore pass to a stabilization $(\wh\D,\wh\A)$ and prove (F3) stated in the Introduction.
We deform $\wh\A$ locally near $Z_\a$ to satisfy Conditions (C1)--(C3) in Section~\ref{ILT} and to be convenient for computing the relevant eigenspaces.
Theorem~\ref{M2ILT} then expresses $\ind\wh\D=\ind\D$ as the sum of the contributions
from the connected components of $Z_{\wh\A}=Z_\a\sqcup C$.
By construction of the stabilization, the contribution from $C$ to $\ind\wh\D$ is given by Proposition~\ref{CalC}.
We apply Corollary~\ref{ILTPoint} to compute the contribution from $Z_\a$.

\subsection{Stabilization}
\label{Sta}

We make the following stabilization.

\begin{defn}\label{D:Sta}
Set $\wh W=\wh W^+\oplus \wh W^-$, where $\wh W^\pm=W^\pm\oplus \ul{\C}^3$, and equip the trivial summands with their trivial connections.
Extend $\psi$ to
\begin{equation}\label{st-bdlmap}
\wh\psi := \psi\oplus |\a|^2 I_3:\wh W^+\to \wh W^-,
\end{equation}
where $I_3$ denotes the identity map between the two trivial summands $\ul{\C}^3$. Set
$$
\wh E:=E_{\wh W},\qquad\wh\D:=\D_{\wh W},\qquad\text{and}\qquad \wh\A:=\A_{\wh \psi},
$$
where $E_{\wh W}$, $\D_{\wh W}$, and $\A_{\wh \psi}$ are defined by \eqref{Tbdle}, \eqref{TDirac}, and \eqref{TPer}.
\end{defn}

The stabilized pair $(\wh\D,\wh\A)$ has the following properties:
\begin{itemize}
\item[(i)]
By construction, the stabilized pair $(\wh\D,\wh\A)$ satisfies the concentration condition~\eqref{CC}.
\item[(ii)]
Since $\ch(\wh W^+ - \wh W^-)=\ch(W^+-W^-)$, by the Atiyah--Singer index theorem,
$$
\ind\wh\D =\ind\D=3c_2(M)-c_1^2(M).
$$
\item[(iii)]
The stabilized bundles and perturbation split as
$$
\wh E^\pm=E^\pm\oplus E_{st}^\pm\qquad \text{and}\qquad\wh\A=\A\oplus\A_{st},
$$
where $E_{st}:=E_{\ul{\C}^3\oplus\ul{\C}^3}=\S\otimes(\ul{\C}^3\oplus\ul{\C}^3)$ with its natural $\Z_2$-grading and
$\A_{st}:=\A_{|\a|^2I_3}:E_{st}^+\to E_{st}^-$
is the perturbation defined by \eqref{TPer}.
Both $E$ and $E_{st}$ are $\Cl(T^*M)$-submodules of $\wh E$.
\item[(iv)]
The singular set of $\wh\A$ remains $Z_{\wh\A}=Z_\a\sqcup C$.
Let $Z$ be a component of $Z_{\wh \A}$.
\begin{itemize}
\item
If $Z$ is a point of $Z_\a$, then $\ker(\wh\A|_Z)=\wh E^+|_Z$ and $\ker(\wh\A^*|_Z)=\wh E^-|_Z$.
\item
If $Z$ is a component of $C$, then $\ker(\wh\A|_Z)=E^+|_Z$ and $\ker(\wh\A^*|_Z)=E^-|_Z$.
\end{itemize}
In either case, since the connection on $\wh E=E\oplus E_{st}$ is the direct sum connection, for $v\in TM|_Z$,
$\nabla_v\wh\A$ and $\nabla_v\wh\A^*$ satisfy \eqref{P4}.
\item[(v)]
When $Z$ is a component of $C$, (iv) shows that the normal derivatives $\wh\A_j$ defined by \eqref{nder} in Section~\ref{ILT}
agree with $\A_j$ given in \eqref{nder-C} in Section~\ref{non-degeneracy}. Therefore, the contribution from
$C$ to $\ind\wh\D$ is given by Proposition~\ref{CalC}.

\end{itemize}

In the remaining subsections, we make suitable local deformations near $Z_\a$, verify Conditions (C1)--(C3) in Section~\ref{ILT},
and compute the localized contributions.

\subsection{Local deformation}

We fix the local frames and coordinates used in the deformation and the eigenspace computation.
For each $p\in Z_\a$, choose an open neighborhood $U_p$ such that the closures of the $U_p$ are pairwise disjoint and disjoint from $C$.
For each $p\in Z_\a$, choose an orthonormal basis $\{e_1,e_2=Je_1,e_3,e_4=Je_3\}$ for $T_pM$ and its dual basis $\{e^i\}$.
Choose coordinates $x=(x_1,\cdots,x_4)$ on $U_p$, centered at $p$,
such that $\pa x_i|_p=e_i$, and extend $\{e_i\}$ by setting $e_i=\pa x_i$.
Lastly, choose a local unitary frame $\{\t^1,\t^2\}$ for $\L^{1,0}M$ such that, at $p$,
\begin{equation*}
\t^1=\fdfrac{1}{\sqrt{2}}(e^1+ie^2)
\qquad\text{and}\qquad\t^2=\fdfrac{1}{\sqrt{2}}(e^3+ie^4).
\end{equation*}

Let $p\in Z_\a$. Identify $U_p$ with the unit ball in $\C^2$, and set $z_1=x_1+ix_2$ and $z_2=x_3+ix_4$.
Let $\ep_p\in\{1,-1\}$ denote the local index of the transverse zero of $\a$ at $p$.
By Appendix~\ref{homotopy}\,(a), after a deformation supported in $U_p$, we may assume that
$$
\a=\left\{
\begin{array}{ll}
z_1\t^1+z_2\t^2
&\text{if}\ \ep_p=1,\\[3pt]
\ov{z}_1\t^1+z_2\t^2
&\text{if}\ \ep_p=-1.
\end{array}\right.
$$
Since $\eta$ is nowhere zero on $U_p$, it has the form $\eta=f\, \theta^1\wedge \theta^2$, where $f:U_p\to \C^*$.
Hence, after deformation supported in $U_p$ we can also assume that
$\eta= \theta^1\wedge \theta^2$.

\subsubsection{Local matrix representation}

Let $\{\mu_1,\mu_2\}$ denote the unitary frame for $T^{1,0}M$ dual to $\{\t^1,\t^2\}$.
Choose a unitary frame for $W^-=\sl(\L^{1,0}M)$ over $U_p$, where
$$
B_1:=\mu_1\otimes\t^2,
\quad
B_2:=\mu_2\otimes\t^1,
\quad
B_3:=\fdfrac{1}{\sqrt{2}}\big(\mu_1\otimes\t^1-\mu_2\otimes\t^2\big).
$$
We first consider the case $\ep_p=1$. Locally we have
$$
\a^\sharp=\ov{z}_1\mu_1+\ov{z}_2\mu_2,\quad
\a_\eta=-\ov{z}_2\t^1+\ov{z}_1\t^2,\quad
\a_\eta^\sharp=-z_2\mu_1+z_1\mu_2.
$$
With respect to the above unitary frames, the components of the bundle map $\psi$ in \eqref{bdle-map} are
$$
\psi(\t^1,0)=z_1B_2 -\fdfrac{\sqrt{2}}{2}z_2B_3,\quad
\psi(\t^2,0)=-z_2B_1-\fdfrac{\sqrt{2}}{2}z_1B_3,\quad
\psi(0,1)=\ov{z}_1^2 B_1-\ov{z}_2^2B_2-\sqrt{2}\ov{z}_1\ov{z}_2 B_3,
$$
and hence the matrix representation of $\psi$ is
\begin{equation}\label{MR}
\Psi_+(z_1,z_2)=
\begin{bmatrix*}[r]
0 & -z_2 & \ov{z}_1^2 \\[3pt]
z_1 & 0 & -\ov{z}_2^2\\[3pt]
-\fdfrac{\sqrt{2}}{2}z_2 & -\fdfrac{\sqrt{2}}{2}z_1
& -\sqrt{2}\ov{z}_1\ov{z}_2
\end{bmatrix*}.
\end{equation}
If $\ep_p=-1$, the matrix representation $\Psi_-$ of $\psi$ has the same form with $z_1$ replaced by $\ov{z}_1$. Thus
$$
\Psi_-=\Psi_+\circ\rho,
$$
where $\rho:\C^2\to\C^2$ denotes the involution defined by $\rho(z_1,z_2)=(\ov{z}_1,z_2)$.

\subsubsection{The relative homotopy class}
\label{RHC}

Note that $\Psi_+$ is invertible on $S^3$ since
$$
\det\Psi_+(z_1,z_2)=-\frac{1}{\sqrt2}\bigl(|z_1|^2+|z_2|^2\bigr)^2.
$$
Since $\Psi_-=\Psi_+\circ \rho$ and $\rho$ preserves $S^3$, the restrictions $\Psi_\pm|_{S^3}$ take values in $GL(3,\C)$,
and therefore so does $\psi|_{S^3}$.
Similarly, the restriction $\wh\psi|_{S^3}$ of the bundle map defined in \eqref{st-bdlmap}
takes values in $GL(6,\C)$ (note that $|\a|=1$). The homotopy classes of these maps are determined by the integer-valued invariant $\tau_n$
defined in \eqref{winding}. Specially, we have:

\begin{lemma}\label{mrep}
$\tau_6(\wh\psi|_{S^3})=\tau_3\big(\psi|_{S^3}\big)=-3\ep_p$.
\end{lemma}

\begin{proof}
Let $(u,v_1,v_2)$ denote the Hopf coordinates on $S^3$ given by
$$
z_1=e^{i v_1}\sin u\quad\text{and}\quad z_2=e^{i v_2}\cos u,
$$
where $0\leq u\leq \frac{\pi}{2}$ and $0\leq v_1,v_2\leq 2\pi$.
These coordinates are positively oriented.
A direct matrix calculation shows that on $S^3$,
\begin{align*}
\tr\bigl((\Psi_+^{-1}d\Psi_+)^3\bigr)&=
3\tr\left( (\Psi_+^{-1}\pa_u\Psi_+)\left[ \Psi_+^{-1}\pa_{v_1}\Psi_+,
\Psi_+^{-1}\pa_{v_2}\Psi_+ \right]\right) du\wedge dv_1\wedge dv_2
\\[3pt]
&= 36\sin u\cos u\,du\wedge dv_1\wedge dv_2.
\end{align*}
Using \eqref{winding}, it follows that $\tau_3\big(\Psi_+|_{S^3}\big)=-3$.
Since $\Psi_-=\Psi_+\circ \rho$ and $\rho|_{S^3}$ has degree $-1$,
$\tau_3\big(\Psi_-|_{S^3}\big)=3$.
Finally, $\wh\psi=\psi\oplus |\a|^2I_3$ with $|\a|=1$ on $S^3$, so the first equality of the lemma follows from
Appendix\,(e).
\end{proof}

We can now put $\wh\psi$ into a canonical local form.
Let $\Phi: \C^2\to M_2(\C)$ be the map defined by formula \eqref{identification}.
Define a map $R:U_p\to M_6(\C)$ by
\begin{equation}\label{DefmapR}
R:=
\left\{
\begin{array}{cl}
(\Phi\circ\rho)^{\oplus 3} &\text{if }\ep_p=1,\\[3pt]
\Phi^{\oplus 3} &\text{if }\ep_p=-1.
\end{array}
\right.
\end{equation}
Then $R$ is invertible on $U_p\setminus\{p\}$, and $\tau_6\big( R|_{S^3}\big)=-3\ep_p$ by  Appendix~\ref{homotopy}\,(e),(f).
Lemma~\ref{mrep} then shows that
$\tau_6\big(\wh\psi|_{S^3}\big)=\tau_6\big(R|_{S^3}\big)$.
Since $\tau_6:\pi_3(GL(6,\C))\to\Z$ is an isomorphism, we have
\begin{equation*}
\big[\wh\psi|_{S^3}\bigr]=\bigl[R|_{S^3}\big]
\quad\text{in}\quad
\pi_3(GL(6,\C)).
\end{equation*}
Therefore, by Appendix~\ref{homotopy}\,(b), we may assume that
\begin{equation}\label{psi=R}
\wh\psi=R\quad\text{near $p$}
\end{equation}
after a deformation supported in $U_p$.

\subsubsection{The deformed perturbation}

After the deformation \eqref{psi=R}, we can explicitly calculate the matrix of the perturbation $\wh\A$  at a point $p\in Z_\a$,
and verify Conditions~(C1)--(C3) of Section~\ref{ILT}.
By construction, the deformed perturbation $\wh\A$ still satisfies property~(iv) in Section~\ref{Sta}, and thus
Condition~(C1) in Section~\ref{ILT} holds.

Next note that
$$
\Phi(z)=x_1A_1+x_2A_2+x_3A_3+x_4A_4,\qquad\Phi\circ \rho(z)=x_1A_1-x_2A_2+x_3A_3+x_4A_4,
$$
where
$$
A_1=\begin{bmatrix} 1 & 0 \\ 0 & 1 \end{bmatrix},\qquad
A_2=\begin{bmatrix*}[r] i & 0 \\ 0 & -i \end{bmatrix*},\qquad
A_3=\begin{bmatrix*}[r] 0 & 1 \\ -1 & 0 \end{bmatrix*},\qquad
A_4=\begin{bmatrix} 0 & i \\ i & 0 \end{bmatrix}.
$$
The matrices $A_2,A_3,A_4$ satisfy the quaternionic relations. The map \eqref{DefmapR} is $R=\sum_{j=1}^4 x_jR_j$ for the $6\times 6$ matrices
$$
R_1:=A_1^{\oplus 3},\qquad  R_2:=(-\ep_pA_2)^{\oplus 3},\qquad  R_3:=A_3^{\oplus 3},\qquad  R_4:=A_4^{\oplus 3},
$$
and the matrices $\wh\A_j$ and $\wh\A^*_j$ (defined in Section~\ref{ILT}) for $Z=p$ are
$$
\wh\A_j=\begin{bmatrix} \ga\otimes R_j & 0 \\[3pt] 0 & \ga\otimes R_j^* \end{bmatrix}:\wh E^+_p\to \wh E^-_p,\qquad \wh\A^*_j=(\wh\A_j)^*.
$$
It follows that $\wh\A_j^*\wh\A_k+\wh\A_k^*\wh\A_j=2\d_{jk}Q_p$ and $Q_p=Id$.
Therefore, Condition~(C2) holds, and Condition~(C3) holds with $d=1$ and $\la_1=1$.

\subsection{Simultaneous eigenspaces}

Let $M_j^\pm:\wh E_p^\pm\to \wh E_p^\pm$ denote the commuting hermitian endomorphisms defined by \eqref{CCoperator} for $Z=p$,
and let $S_p^\pm$ be as in Theorem~\ref{M2ILT}.
Since $\ga|_{\S^\pm}=\pm Id$, $R_1=I_6$, and $R_j^*=-R_j$ ($j\geq 2$), we have
$$
M_1^+=M_1^-=\begin{bmatrix} 0 & \c(e^1)\otimes I_6 \\[3pt] -\c(e^1)\otimes I_6 & 0 \end{bmatrix},\quad
M_j^+=-M_j^-=-\begin{bmatrix} 0 & \c(e^j)\otimes R_j \\[3pt] \c(e^j)\otimes R_j & 0 \end{bmatrix}\quad(j\geq 2).
$$

The following proposition completes the proof of (F3) in the Introduction.

\begin{prop}\label{CalC0}
The contribution from $Z_\a$ to $\ind\wh\D$ is $3c_2(M)$.
\end{prop}

\begin{proof}
Since $Q_p=Id$ with $d=1$ and $\la_1=1$, the proof of Corollary~\ref{ILTPoint} shows that
$$
S_p^\pm=\underset{j}{{\textstyle \bigcap}}\,\ker(M_j^\pm-Id).
$$
To compute $\dim S_p^\pm$, we use Clifford multiplication~\eqref{CliffM}.
The equation $M_1^+\zeta=\zeta$ shows that any $\zeta\in S_p^+$ has the form
$$
\zeta=\big(-1\otimes u + \vt^1\w\vt^2\otimes v, \vt^1\otimes u+\vt^2\otimes v\big),
$$
for some $u,v\in \C^6$. The remaining equations $M_j^+\zeta=\zeta$ ($j\geq 2$) are equivalent to
$$
R_2u=-iu,\quad R_2v=iv,\qquad R_3u=v,\quad R_3v=-u,\qquad R_4u=-iv,\quad R_4v=-iu.
$$
Since $R_2R_3=-\ep_pR_4$, the solutions $(u,v)$ are characterized as follows:
$$
R_2u=-iu,\ v=R_3u\quad \text{when }\ep_p=1,\quad\text{and}\quad
u=v=0\quad \text{when }\ep_p=-1.
$$
Since the $(\pm i)$-eigenspaces of $R_2$ have dimension $3$, it follows that $\dim S^+_p$ is $3$ if $\ep_p=1$ and $0$ if $\ep_p=-1$.
By the same argument, using $M_1^+=M_1^-$ and $M_j^-=-M_j^+$ ($j\geq 2$), $\dim S^-_p$ is $0$ if $\ep_p=1$ and $3$ if $\ep_p=-1$.
Therefore, by Corollary~\ref{ILTPoint}, the contribution from $Z_\a$ is
\begin{equation*}
\sum_{p\in Z_\a} \ind_p =\sum_{p\in Z_\a} \big(\dim S^+_p- \dim S^-_p\big)=3\sum_{p\in Z_\a} \ep_p =3c_2(M).
\end{equation*}
\end{proof}

By (v) in Section~\ref{Sta}, the contribution from $C$ to $\ind\wh\D$ is $-c_1^2(M)$.
Together with this, Proposition~\ref{CalC0} gives $\ind\D=3c_2(M)-c_1^2(M)$,
completing the computation of all localized contributions.

\appendix

\renewcommand{\theequation}{\thesection.\arabic{equation}}
\setcounter{equation}{0}

\section{Relative homotopy classes}
\label{homotopy}

In this appendix, we recall several standard facts about relative homotopy classes of maps; see, e.g., \cite{H}.
\begin{itemize}
\item[(a)]
By the long exact sequence of relative homotopy groups,
$$
\pi_4(\C^2,\C^{2}-\{0\})\cong \pi_3(S^3)\cong \Z,
$$
where the first isomorphism is induced by the boundary homomorphism given by restriction to $S^3$, and the second is given by the degree.
\item[(b)]
Similarly, for $n\geq 2$,
$$
\pi_4(M_n(\C),GL(n,\C))\cong \pi_3(GL(n,\C))\cong \Z,
$$
where the first isomorphism is the boundary homomorphism, and the second isomorphism is described below.
\end{itemize}

\non
The remaining facts provide an explicit isomorphism $\pi_3(GL(n,\C))\cong \Z$ for $n\geq 2$.

\begin{itemize}
\item[(c)]
Let $X:GL(n,\C)\to M_n(\C)$ be the inclusion $X(A)=A$, and let $\T_n=X^{-1}dX$ denote the Maurer--Cartan form on $GL(n,\C)$. Set
$$
\Om_n:=-\sdfrac{1}{24\pi^2}\tr_n(\T_n^3).
$$
By the Maurer--Cartan equation $d\T_n+\T_n^2=0$ and the graded commutativity of differential forms,
$$
d\tr_n(\T_n^3) =3\tr_n(d\T_n \T_n^2) =-3\tr_n(\T_n^4)=0.
$$
Thus $\Om_n$ is closed.

\item[(d)]
The block inclusion $j_n:SU(2)\hookrightarrow GL(n,\C)$ given by
$j_n(A)= A\oplus I_{n-2}$ induces an isomorphism
$$
(j_n)_*:\pi_3(SU(2))\cong \pi_3(GL(n,\C))\qquad\text{such that}\qquad j_n^*\T_n=(j_2^*\T_2)\oplus0,\quad j_n^*\Om_n=j_2^*\Om_2.
$$
The isomorphism follows from the long exact sequences of homotopy groups associated to the fiber bundles
$$
SU(k)\hookrightarrow SU(k+1)\xr{A\mapsto A_{k+1}} S^{2k+1}\quad(2\leq k\leq n-1),\qquad SU(n)\hookrightarrow U(n)\xr{\det}S^1,
$$
where $A_{k+1}$ is the last column of $A$, together with the facts that
$\pi_3(S^m)=0$ for $m\ne 2,3$, $\pi_4(S^m)=0$ for $m\ne 2,3,4$,  and $U(n)$ is a deformation retract of $GL(n,\C)$.

\item[(e)]
For a smooth map $g:S^3\to GL(n,\C)$, define
\begin{equation}\label{winding}
\tau_n(g):=\int_{S^3}g^*\Om_n=-\sdfrac{1}{24\pi^2}\int_{S^3}\tr_n\big((g^{-1}dg)^3\big).
\end{equation}
Since $\Om_n$ is closed, $\tau_n(g)$ depends only on the homotopy class $[g]$.
Moreover, if $g_i:S^3\to GL(n_i,\C)$ $(i=1,2)$ are smooth maps, then
$$
\tau_{n_1+n_2}(g_1\oplus g_2)=\tau_{n_1}(g_1)+\tau_{n_2}(g_2).
$$
In particular, $\tau_{n+r}(g\oplus I_r)=\tau_n(g)$.

\item[(f)]
Consider the standard identification
\begin{equation}\label{identification}
\Phi:S^3\to SU(2)\quad\text{given by}\quad
\Phi(z_1,z_2)=
\begin{bmatrix}
z_1&z_2\\
-\ov{z}_2&\ov{z}_1
\end{bmatrix}.
\end{equation}
For every smooth map $h:S^3\to SU(2)$, we have
$$
\tau_2(h)=\deg(\Phi^{-1}\circ h).
$$
See Section 6.4 of \cite{Nab}. Consequently, $\tau_2:\pi_3(SU(2))\to\Z$ is an isomorphism.
In particular, $\tau_2(\Phi)=1$.

\item[(g)]
By (d), $\tau_n(j_n\circ h)=\tau_2(h)$. Hence, since $(j_n)_*$ is an isomorphism,
$\tau_n$ takes values in $\Z$, and we obtain the commutative diagram
$$
\xymatrix{
\pi_3(SU(2)) \ar[rr]^{(j_n)_*} \ar[dr]_{\tau_2} && \pi_3(GL(n,\C)) \ar[dl]^{\tau_n}\\
&\Z
}
$$ In particular,
$\tau_n:\pi_3(GL(n,\C))\to\Z$ is an isomorphism.

\end{itemize}


\medskip

\non
School of Data, Mathematical, and Statistical Sciences, University of Central Florida, Orlando, FL 32816

\non
Email: junho.lee\@@ucf.edu


\begin{thebibliography}{99}






\bibitem{GW} C. Gerig and C. Wendl, {\em Generic transversality for unbranched covers of closed pseudoholomorphic curves},
Comm. Pure Appl. Math. {\bf 70} (2017), no. 3, 409–443.

\bibitem{H} A. Hatcher, {\em Algebraic Topology}, Cambridge University Press, Cambridge, 2002.





\bibitem{L} J. Lee, {\em Note on concentration via the conjugate-linear Hodge star operator}, Glas. Mat. {\bf 61} (2026), no. 1, 151--159.

\bibitem{LP} J. Lee and T. Parker, {\em  Spin Hurwitz numbers and the Gromov-Witten invariants of Kähler surfaces},
Comm. Anal. Geom. {\bf 21} (2013), no. 5, 1015–1060.

\bibitem{M1} M. Maridakis, {\em Spinor pairs and the concentration principle for Dirac operators},
Trans. Amer. Math. Soc. {\bf 369} (2017), no. 3, 2231--2254.


\bibitem{M2}  M. Maridakis, {\em A localization theorem for Dirac operators}, arXiv:2110.11654.



\bibitem{Nab} G. Naber, {\em Topology, Geometry and Gauge Fields: Interactions}, 2nd ed., Applied Mathematical Sciences 141, Springer, 2011.







\bibitem{Nag} A. Nagy, {\em  Conjugate linear perturbations of Dirac operators and Majorana fermions},
J. Geom. Anal. {\bf 35} (2025), no. 6, Paper No. 169.

\bibitem{Ni} L. Nicolaescu, {\em Lectures on the geometry of manifolds}, 3rd ed., World Scientific, Hackensack, NJ, 2021.



\bibitem{P} G. Parker, {\em Concentrating Dirac operators and generalized Seiberg-Witten equations},
Math. Res. Lett. {\bf 33} (2026), no. 1, 1–55.




\bibitem{PR} I. Prokhorenkov and K. Richardson, {\em Perturbations of Dirac operators},
J. Geom. Phys. {\bf 57} (2006), no. 1, 297--321.

\bibitem{R} D. Rauch, {\em Perturbations of the d-bar operator}, Doctoral dissertation, Harvard University, 2004.
















\bibitem{T1} C. Taubes, {\em Counting pseudo-holomorphic submanifolds in dimension 4}, J. Diff. Geom. {\bf 44} (1996),
no. 4, 818--893.

\bibitem{T2} C. Taubes, {\em The Seiberg-Witten and Gromov invariants}, Math. Res. Lett. {\bf 2} (1995), no. 2, 221–238.


\bibitem{W} E. Witten, {\em Supersymmetry and Morse theory}, J. Diff. Geom.
{\bf 17} (1982), no. 4,  661--692.








\end{thebibliography}
\end{document}